\documentclass[pdflatex,sn-mathphys-num]{sn-jnl}

\usepackage{graphicx}%
\usepackage{multirow}%
\usepackage{amsmath,amssymb,amsfonts}%
\usepackage{amsthm}%
\usepackage{mathrsfs}%
\usepackage[title]{appendix}%
\usepackage{xcolor}%
\usepackage{textcomp}%
\usepackage{manyfoot}%
\usepackage{booktabs}%
\usepackage{algorithm}%
\usepackage{algorithmicx}%
\usepackage{algpseudocode}%
\usepackage{listings}%
\usepackage{tikz}%
\usepackage{hyperref}%

\theoremstyle{thmstyleone}%

\newtheorem{theorem}{Theorem}[section]
\newtheorem{corollary}[theorem]{Corollary}
\newtheorem{proposition}[theorem]{Proposition}
\newtheorem{lemma}[theorem]{Lemma}
\newtheorem{remark}[theorem]{Remark}
\newtheorem{remarks}[theorem]{Remarks}
\theoremstyle{thmstyletwo}%
\theoremstyle{thmstylethree}%

\begin{document}

\title[Elastic flow of planar curves inside cones]{Length-constrained, length-penalised and free elastic flows of planar curves inside cones}


\author[1,2]{\fnm{Mashniah A.} \sur{Gazwani}}\email{magazwani@iau.edu.sa}
\author[2]{\fnm{James A.} \sur{McCoy}}\email{James.McCoy@newcastle.edu.au}

\affil[1]{\orgdiv{Mathematics}, \orgname{Imam Abdulrahman Bin Faisal University}, \orgaddress{\street{Prince Naif Road}, \city{Dammam}, \postcode{32256}, \state{Eastern Province}, \country{Saudi Arabia}}}

\affil[2]{\orgdiv{School of Information and Physical Sciences}, \orgname{The University of Newcastle}, \orgaddress{\street{University Drive}, \city{Callaghan}, \postcode{2308}, \state{New South Wales}, \country{Australia}}}


\abstract{We study families of smooth, embedded, regular planar curves
 $ \alpha : \left [-1,1  \right ]\times \left [0,T  \right )\to \mathbb{R}^{2}$ with generalised Neumann boundary conditions inside cones, satisfying three variants of the fourth-order nonlinear $L^2$- gradient flow for the elastic energy: (1) elastic flow with a length penalisation, (2) elastic flow with fixed length and (3) the unconstrained or `free' elastic flow.  Assuming neither end of the evolving curve reaches the cone tip, existence of smooth solutions for all time given quite general initial data is well known, but classification of limiting shapes is generally not known.  For cone angles not too large and with suitable smallness conditions on the $L^2$-norm of the first arc length derivative of curvature of the initial curve, we prove in cases (1) and (2) smooth exponential convergence of solutions in the $C^\infty$-topology to particular circular arcs, while in case (3), we show smooth convergence to an expanding self-similar arc. }




\maketitle

\section{Introduction}\label{sec1}

The elastic energy has been extensively studied since the Bernoulli model of an elastic rod.  We refer the reader to \cite{T83} for an interesting historical account with reference to impacts of elasticity on mathematical analysis.  Critical points of the elastic energy are called \emph{elasticae}.  More recently, the corresponding $L^2$-gradient flow for the elastic energy has generated much interest, along with related fourth-order curvature flows.  For a review, we refer the reader to the introduction of \cite{WW}.  The foci of that article were the curve diffusion flow and the elastic flow for planar curves between parallel supporting lines with generalised Neumann boundary conditions; given the fourth-order nature of these flows, it is natural to supplement the classical Neumann boundary condition with a `no-curvature-flux' condition on the boundary, although other boundary conditions are also possible (see, for examples, \cite{lin2012l2, DP14, DPS16, DLP17, DD24}).  Recently the present authors considered planar curves with the same boundary conditions evolving by the curve diffusion flow inside cones \cite{GM24}.  There it was showed that initial curves suitably far from the tip of the cone and with sufficiently small oscillation of curvature give rise to solutions of the curve diffusion flow that exist for all time and converge to circular arcs that, together with the cone boundary, enclose the same area as the initial curve.  It is natural to consider the elastic flow in the same setting.  In contrast to the curve diffusion flow, for which circular arcs are stationary solutions, for the unconstrained elastic flow it is easy to see that circular arcs move outwards self-similarly.  Thus one might suspect more generally solutions might approach these self-similar solutions.  We provide sufficient conditions on the cone angle and initial curve that guarantee this is indeed the case.  We also consider two ways in which the elastic flow can be modified to control the length of the evolving curve, similarly as was done in \cite{DKS02} in the case of closed curves.  In these cases, with sufficiently small cone angle and small initial oscillation of curvature we obtain long time convergence to a stationary limit curve which is a circular arc.  It is interesting to note that while in the case of curve diffusion, the relevant smallness quantity was the oscillation of curvature, whereas here it is the $L^2$-norm of the first derivative of curvature with respect to arc length.

The structure of this article is as follows.  In Section \ref{S:prelim} we set up notation, fully describe our setting and provide some known results that are needed in the subsequent analysis. In Sections \ref{S:lp}, \ref{S:lf} and \ref{S:f}, we consider elastic flows respectively with (1) a length penalty term ($\lambda>0$), (2) a length constraint ($\lambda = \lambda(t)$ for a specifically chosen function $\lambda(t)$)  and (3) no constraint ($\lambda\equiv 0$, the so-called free elastic flow).

The authors would like to thank Glen Wheeler for useful discussions.
\section{Preliminaries} \label{S:prelim}
 \begin{figure}[H] 
  \centering
  \tikzset{every picture/.style={line width=0.75pt}}
\begin{tikzpicture}[x=0.75pt,y=0.75pt,yscale=-1,xscale=1]
\draw    (54.33,64) -- (200,213) ;
\draw    (334.33,52) -- (200,213) ;
\draw [color={rgb, 255:red, 110; green, 168; blue, 236 }  ,draw opacity=1 ]   (127.17,138.5) .. controls (174.33,108) and (164.33,19) .. (267.17,132.5) ;
\draw [color={rgb, 255:red, 208; green, 2; blue, 27 }  ,draw opacity=1 ][fill={rgb, 255:red, 12; green, 29; blue, 250 }  ,fill opacity=1 ]   (267.17,132.5) -- (294.95,161.55) ;
\draw [shift={(296.33,163)}, rotate = 226.28] [color={rgb, 255:red, 208; green, 2; blue, 27 }  ,draw opacity=1 ][line width=0.75]    (10.93,-3.29) .. controls (6.95,-1.4) and (3.31,-0.3) .. (0,0) .. controls (3.31,0.3) and (6.95,1.4) .. (10.93,3.29)   ;
\draw [color={rgb, 255:red, 208; green, 2; blue, 27 }  ,draw opacity=1 ]   (267.17,132.5) -- (294.07,99.55) ;
\draw [shift={(295.33,98)}, rotate = 129.23] [color={rgb, 255:red, 208; green, 2; blue, 27 }  ,draw opacity=1 ][line width=0.75]    (10.93,-3.29) .. controls (6.95,-1.4) and (3.31,-0.3) .. (0,0) .. controls (3.31,0.3) and (6.95,1.4) .. (10.93,3.29)   ;
\draw    (252.33,133) -- (260.33,140) ;
\draw    (252.33,133) -- (260.33,125) ;
\draw    (136.33,147) -- (144.33,140) ;
\draw    (136.33,131) -- (144.33,140) ;
\draw [color={rgb, 255:red, 245; green, 166; blue, 35 }  ,draw opacity=1 ]   (171,183) .. controls (211,153) and (248.33,199) .. (254.33,212) ;
\draw [color={rgb, 255:red, 126; green, 211; blue, 33 }  ,draw opacity=1 ]   (211.33,200) .. controls (219.33,202) and (223.33,203) .. (226.33,213) ;
\draw [color={rgb, 255:red, 208; green, 2; blue, 27 }  ,draw opacity=1 ]   (127.17,138.5) -- (98.85,162.7) ;
\draw [shift={(97.33,164)}, rotate = 319.48] [color={rgb, 255:red, 208; green, 2; blue, 27 }  ,draw opacity=1 ][line width=0.75]    (10.93,-3.29) .. controls (6.95,-1.4) and (3.31,-0.3) .. (0,0) .. controls (3.31,0.3) and (6.95,1.4) .. (10.93,3.29)   ;
\draw [color={rgb, 255:red, 208; green, 2; blue, 27 }  ,draw opacity=1 ]   (127.17,138.5) -- (95.71,105.45) ;
\draw [shift={(94.33,104)}, rotate = 46.42] [color={rgb, 255:red, 208; green, 2; blue, 27 }  ,draw opacity=1 ][line width=0.75]    (10.93,-3.29) .. controls (6.95,-1.4) and (3.31,-0.3) .. (0,0) .. controls (3.31,0.3) and (6.95,1.4) .. (10.93,3.29)   ;

\draw (340,30.4) node [anchor=north west][inner sep=0.75pt]    {$\overline{\gamma }_{2}$};
\draw (301,143.4) node [anchor=north west][inner sep=0.75pt]  [color={rgb, 255:red, 208; green, 2; blue, 27 }  ,opacity=1 ]  {$e_{+}$};
\draw (297,96.4) node [anchor=north west][inner sep=0.75pt]  [color={rgb, 255:red, 208; green, 2; blue, 27 }  ,opacity=1 ]  {$\nu $};
\draw (191,48.4) node [anchor=north west][inner sep=0.75pt]  [font=\Large,color={rgb, 255:red, 74; green, 146; blue, 226 }  ,opacity=1 ]  {$\alpha $};
\draw (38,206.4) node [anchor=north west][inner sep=0.75pt]    {$----------------------$};
\draw (189,151.4) node [anchor=north west][inner sep=0.75pt]    {$\textcolor[rgb]{0.96,0.65,0.14}{\theta _{1}{}}$};
\draw (226.07,192.33) node [anchor=north west][inner sep=0.75pt]  [rotate=-0.45]  {$\textcolor[rgb]{0.49,0.83,0.13}{\theta _{2}}$};
\draw (42,35.4) node [anchor=north west][inner sep=0.75pt]    {$\overline{\gamma }_{1}$};
\draw (79,103.4) node [anchor=north west][inner sep=0.75pt]  [color={rgb, 255:red, 208; green, 2; blue, 27 }  ,opacity=1 ]  {$\nu $};
\draw (76.52,142.91) node [anchor=north west][inner sep=0.75pt]  [color={rgb, 255:red, 26; green, 100; blue, 187 }  ,opacity=1 ,rotate=-356.68]  {$\textcolor[rgb]{0.82,0.01,0.11}{e_{-}}$};

\end{tikzpicture}
\caption[The set-up.]
 \label{fig1}
\end{figure}

Given an initial smooth, immersed, regular open planar curve with boundary $\alpha_0: \left[ -1, 1\right] \times \left[0, T\right) \rightarrow \mathbb{R}^2$, we consider the family of curves $\alpha\left( \cdot, t\right)$ moving with normal velocity $F$: \begin{equation}\label{f}
    \frac{\partial }{\partial t}\alpha  =-\,F[k]\,\nu,
\end{equation}
where $F[k]$ indicates the normal speed of the curve and $\alpha( \cdot, 0) = \alpha_0$.
Above $\nu$ is the outer unit normal vector field to the curve $\alpha$, and $k=\left<\kappa, \nu  \right>=-\left<  \alpha_{ss}, \nu \right >$  is the scalar curvature of $\alpha$, where $\left<  \cdot, \cdot \right >$ is the usual Euclidean inner product in $\mathbb{R}^2$ and $s$ is the arc-length parameter of curve $\alpha$, defined via
\begin{equation} \label{s}
  s\left( u, t\right) = \int_{-1}^{u} \left| \alpha_u \left( \tilde u, t\right) \right| d\tilde u \mbox{.}
\end{equation}
The sign in \eqref{f} is chosen such that the equation is parabolic in the generalised sense.  Let us further denote by $\tau =\displaystyle{\frac{\alpha_{u} }{\left|\alpha _{u} \right|}}=\alpha _{s}$ the unit tangent vector field along $\alpha$.  

The three speeds $F$ of interest in this article are
\begin{enumerate}
    \item[\textnormal{(1)}] $F= -k_{ss} - \frac{1}{2} k^3 + \lambda\, k$ for any constant $\lambda>0$ (the `length-penalised elastic flow');
    \item[\textnormal{(2)}] $F= -k_{ss}  - \frac{1}{2} k^3 +\lambda(t) k$ for a specific function $\lambda(t)$ chosen such that the length of the evolving curve is fixed (a `length-constrained elastic flow');
    \item[\textnormal{(3)}] $F= - k_{ss}  - \frac{1}{2} k^3$ the so-called `free elastic flow'.
\end{enumerate}
These speeds were considered in the case of closed curves in \cite{DKS02}, with longtime existence and subconvergence to elastica considered in cases (1) and (2).  Very recently Miura and Wheeler \cite{TW24} considered case (3) for closed curves, introducing a new monotone quantity that facilitates exponential convergence to expanding circles for initial curves that are sufficiently geometrically close to circles.  It turns out that this quantity may also be used to prove the analogous result here: convergence to expanding circular arcs.

The boundary conditions for our evolving planar curve comprise the classical Neumann condition and a no-curvature-flux condition.  Fix $0\leq \theta_2 < \theta_1 < 2\pi$ and let $\gamma_{i} :\left [ 0,\infty  \right )\to \mathbb{R}^{2}~ (i=1,2)$, such that $ \gamma_{i}\left ( \rho  \right )=\rho \, \theta _{i},$ and denote the image sets by 
$$\bar{\gamma_{i}} :\left\{\left ( x,y \right )\in \mathbb{R}^{2}: \left ( x,y \right )=\left ( \rho \cos\theta _{i}, \rho\sin \theta _{i}\right ), \rho > 0 \right\}.$$
The set $\bar{\gamma_{1}} \cup \bar{\gamma_{2}}$ forms the boundary for our evolving curve, namely a cone in $\mathbb{R}^2$ with the tip at the origin and cone angle $\theta_1 - \theta_2 < 2\pi$.  We remark that in our recent work \cite{GM24} on the curve diffusion flow, we required the cone angle to be strictly less than $\pi$ in order to have a preserved smallness condition on the oscillation of curvature under the flow, facilitating existence for all time of a solution and convergence to a circular arc.  There is no such requirement here to have long time existence; this will be true for any cone angle up to $2\pi$ and follows by very similar arguments as in the case of closed curves \cite{DKS02}.  However, the classification of asymptotic behaviour does require a smallness condition on the cone angle and associated smallness condition on the oscillation of curvature.

The interior of our initial curve $\alpha_0: \left[ -1, 1\right] \rightarrow \mathbb{R}^2$ is assumed to be contained in the region 
$$\left\{ \left( x, y\right) \in \mathbb{R}^2: \left( x, y\right) = \left( \rho \cos \theta, \rho \sin \theta\right), \rho >0, \theta_2 < \theta < \theta_1 \right\} \mbox{,}$$
 that is, `inside' the cone and such that $\alpha_{0} \left ( -1 \right )\in \bar{\gamma _{1}}$ and $ \alpha_{0} \left ( 1 \right )\in \bar{\gamma _{2}}$.  Moreover, $\alpha_0$ is assumed to meet the boundary at each endpoint perpendicularly (the classical Neumann boundary condition) and be such that the curvature $k_0$ of $\alpha_0$, computed using one-sided derivatives, satisfies
$$\left( k_0\right)_s\left( -1\right) = \left( k_0\right)_s\left( 1 \right) = 0 \mbox{,}$$
the \emph{no-curvature-flux condition} on the boundary.  We do not include the cone tip itself in the specification of the boundary condition, as the normal to the boundary there is not well defined.  We now let the curve evolve under \eqref{f} with \emph{generalised Neumann boundary conditions} as long as this is a well-posed problem.  This means that under the evolution, the classical Neumann boundary condition and no curvature flux condition continue to hold.  At least for a short time, by continuity of solutions, the solution will not immediately jump to the cone tip.  It is indeed a key feature of our analysis that we have to avoid the ends of the evolving curve reaching the cone tip, unless they do so in a limiting sense towards the maximal time $T$ of existence of the solution.

Denoting by $e_-$ and $e_+$ unit vectors perpendicular to $\bar{\gamma _{1}}$ and $\bar{\gamma _{2}}$, 
the boundary conditions ensure that the ends of the evolving curve continue to meet either side of the cone that holds as long as the solution to \eqref{f} exists and
 \begin{equation} \label{E:NBC1}
   \left< \nu\left( -1, t\right), e_- \right> = \left< \nu\left( 1, t\right), e_+ \right> = 0
   \end{equation}
   and
   \begin{equation} \label{E:NBC2}
   k_s\left( \pm 1, t\right) = 0 \mbox{.}
   \end{equation}
   By slight abuse of notation, we will denote quantities associated with the evolving curve using the same symbols, be they functions of $u$ and $t$, or $s$ and $t$, as is customary, and should not lead to any confusion.  When we evaluate at endpoints we will denote the spatial argument using $\pm 1$ but spatial derivatives will be respect to $s$ and interpreted in the appropriate one-sided sense.  We denote by $k_{s^{n}}$ the $n$-th iterated derivative of $k$ with respect to arc-length and write $k_{s^{n}}^{2}$ for $\left ( k_{s^{n}} \right )^{2}$.

 Throughout the article, we use $L$ to denote the length of $\alpha\left( \cdot, t\right)$ and write
\begin{equation}\label{eql}
    L\left( t\right) = L[\alpha\left( \cdot, t\right) ]=\int_{\alpha }\, ds.
\end{equation}
All integrals will be over the curve $\alpha$ unless otherwise indicated.  We also have the area $A$ of the region bounded by the curve and cone
\begin{equation} \label{E:area}
  A\left( t\right) = A\left[\alpha\left( \cdot, t\right) \right] = +\frac{1}{2} \int_{\alpha } \left< \alpha, \nu \right> ds \mbox{,}
\end{equation}
where the sign is $+$ due to our choice of outer unit normal.

The average curvature of $\alpha\left( \cdot, t\right)$  is defined as\begin{equation}\label{eqkk}
 \bar{k}[\alpha]=\displaystyle{\frac{1}{L}}\int_{\alpha } k\,ds.   
\end{equation}
Additionally, 
the rotation number of $\alpha\left( \cdot, t\right)$ is
\begin{equation} \label{E:omega}
  \omega \left [ \alpha  \right ] :=\frac{1}{2\pi }\int_{\alpha } k\,ds.
  \end{equation}

This may be thought of as the general definition of rotation number
for any curve for which the integral makes sense.  As the rotation number gives the net turning of the tangent vector to $\alpha$ we observe that in our setting 
$$\omega = \frac{\theta_1 - \theta_2}{2\pi} \mbox{.}$$

To prove our main theorems, we require some well-known fundamental tools.  First are the standard Poincaré-Sobolev-Wirtinger [PSW] inequalities for curves with boundaries.\\

\begin{proposition} \label{psw}
Let $L>0$. Let $g:[0,L]\rightarrow \mathbb{R}$ be an absolutely continuous function with $\int_{0}^{L}g(x) \,dx=0.$ Then 
\begin{equation}\label{eqpsw}
    \int_{0}^{L}g^{2}(x)\, dx\leq \frac{L^{2}}{\pi ^{2}}\int_{0}^{L}\left( g_{x} \right)^{2}(x)\, dx.
\end{equation}
Similarly, if $g:[0,L]\rightarrow \mathbb{R}$ is absolutely continuous and $g(0)=g(L)=0$ then 
\begin{equation*}
\int_{0}^{L}g^{2}(x) \,dx\leq \frac{L^{2}}{\pi ^{2}}\int_{0}^{L} \left( g_{x} \right)^{2}(x)\,dx.
\end{equation*}
\end{proposition}
\begin{proposition} \label{eq266}
If $g:[0,L]\rightarrow \mathbb{R}$ is absolutely continuous and $g(0)=g(L)=0$ then 
\begin{equation}\label{pp}
\left \| g \right \|_{\infty }^{2}\leq \frac{L}{\pi }\int_{0}^{L}\left( g_{x} \right)^2 (x)\, dx.
\end{equation}
Similarly, if  $g:[0,L]\rightarrow \mathbb{R}$ is absolutely continuous and $\int_{0}^{L}g \,dx=0$ then
\begin{equation}\label{eqppsw}
\left \| g \right \|_{\infty }^{2}\leq \frac{2L}{\pi }\int_{0}^{L}\left( g_{x} \right)^2 (x)\, dx.
\end{equation}
\end{proposition}
It is convenient as in \cite{DKS02} and subsequent work to use the following scale-invariant norms: 
 $$\left\| k \right\|_{\ell, p} := \sum_{i=0}^\ell \left\| \partial_s^i k \right\|_p$$
 where
 $$\left\| \partial_s^i k \right\|_p = L^{i+1-\frac{1}{p}} \left( \int_\alpha \left| \partial_s^i k \right|^p ds\right)^{\frac{1}{p}} \mbox{.}$$
 
 The following interpolation inequality for curves with boundary is
 \cite{DP14}[Lemma 4.3].  The analogous inequalities for closed curves appeared earlier in \cite{DKS02}.  We use the standard notation $P_n^m(k)$ to denote a linear combination of terms each of which contain $n$ factors of $k$ with a total of $m$ derivatives with respect to $s$.\\
 
\begin{proposition}\label{p}
  Let $\alpha:\left [ -1, 1 \right ]\to \mathbb{R}^{2}$  be a smooth curve with boundary. Then for any term $P_{n}^{m}\left ( k \right )$ with $n\geq 2$  that contains derivatives of $k$ of order at most $\ell-1$,
  $$\int_\alpha \left|P_{n}^{m}\left ( k \right ) \right| ds\leq c\, L^{\ell-m-n}\left\| k\right\|_{0,2}^{n-p}\left\| k\right\|_{\ell,2}^{p},$$
  where $p=\displaystyle{\frac{1}{\ell}\left ( m+\frac{1}{2} n-1\right )}$ and $c=\left ( \ell,m,n \right ) $. Furthermore, if $m+\frac{n}{2}< 2\ell+1$ then $p<2$ and for any $\varepsilon>0,$
  $$\int_\alpha \left|P_{n}^{m}\left ( k \right ) \right| ds\leq \varepsilon \int  k_{s^{\ell}}^{2}ds+c\,\varepsilon^{\frac{-p}{2-p}}\left( \int k^{2}ds  \right)^{\frac{n-p}{2-p}}+c \left(\int k^{2}ds  \right)^{m+n-1}.$$    
\end{proposition}
Next we record some evolution equations for general normal flow speeds $F$ and a consequence.  The evolution equations are easily derived as, for example, in \cite{WW}.\\

\begin{lemma}\label{T:evlneqns}
  Under the flow \eqref{f} we have the following evolution equations:
\begin{itemize}
\item[\textnormal{(i)}] $\displaystyle{\frac{\partial}{\partial t} \, ds =  -F\, k\, ds}$;
\item[\textnormal{(ii)}]$\displaystyle{\frac{\mathrm{d} }{\mathrm{d} t}L = - \int_{\alpha } k\, F \, ds;}$ 
\item[\textnormal{(iii)}] For each $\ell \in \mathbb{N} \cup \left\{ 0 \right\}$, we have 
$$\displaystyle{\frac{\partial}{\partial t} k_{s^\ell} = F_{s^{\ell + 2}} + \sum_{j=0}^\ell \left( k\, k_{s^{\ell-j}} F\right)_{s^j} \mbox{.}}$$
\item[\textnormal{(iv)}]$\displaystyle{\frac{\mathrm{d} }{\mathrm{d} t} \int_{\alpha } k^2 \, ds =  2 \int_{\alpha } \left( k_{ss} + \frac{1}{2} k^3 \right) F \, ds;}$ 
\item[\textnormal{(v)}]$\displaystyle{\frac{\mathrm{d} }{\mathrm{d} t} \int_{\alpha } k_{s}^2  ds = 2 \int_{\alpha } \left( - k_{s^4} - k^2 k_{ss} + \frac{1}{2} k\, k_s^2 \right) F \, ds.}$ \\
\end{itemize}  
\end{lemma}

Similarly as in \cite{WW}, we have that all odd curvature derivatives are equal to zero on the boundary, as long as the solution to \eqref{eq11} exists for each of the speed functions of interest.  This is critical for future calculations where we `integrate by parts', that is, apply the Divergence Theorem to various functions involving curvature and its derivatives on our curve with boundary.\\

\begin{lemma}\label{T:BCs}
Under the flows \eqref{f}, for each $\ell \in \mathbb{N}$ 
  $$k_{s^{2\ell-1}}\left( \pm 1, t\right)=0 \mbox{.}$$
\end{lemma}
While a smooth solution to the flow \eqref{f} exists, the rotation number $\omega$ is constant.  We state this fact as a Corollary.  The proof is a straightforward calculation, as in \cite{GM24} for example.\\

\begin{corollary} \label{T:omega}
 Under the flow \eqref{f} with generalised Neumann boundary conditions \eqref{E:NBC1} and \eqref{E:NBC2} we have
 \begin{equation}
     \frac{\mathrm{d} }{\mathrm{d} t}\int _{\alpha }k\, ds=0.
 \end{equation}
\end{corollary}
\mbox{}\\

For one of the estimates in Section \ref{S:lf} it will be useful to note that the Euler Lagrange operator for the elastic energy may also be written in divergence form.  It is convenient to introduce some additional notation as in \cite{DKS02}.  For a normal field $\phi$ along curve $\alpha$, we let $\nabla_s \phi$ denote the normal component of $\partial_s \phi$, that is
$$\nabla_s \phi = \partial_s \phi - \left< \partial_s \phi, \tau\right> \tau = \partial_s \phi + \left< \phi, \kappa \right> \tau \mbox{.}$$
We then have\\

\begin{lemma}\label{T:divform}
The Euler Lagrange operator for the elastic energy satisfies
$$\nabla_s^2 \kappa + \frac{1}{2} \left| \kappa \right|^2 \kappa = \partial_s\left( \partial_s \kappa + \frac{3}{2} \left| \kappa \right|^2 \tau \right) = \partial_s \left( -k_s \nu + \frac{1}{2} k^2 \tau \right) = -\left( k_{ss} + \frac{1}{2} k^3 \right) \nu \mbox{.}$$
\end{lemma}
\begin{proof}
The first equality is equation (4.1) in \cite{DKS02} which is a short direct calculation.  The second and third follow also by direct calculation using the Frenet equations and noting our sign convention.    
\end{proof}
Let us remark that short time existence of solutions in this setting is well-known.  We refer the reader to \cite{Wu} for example, where a general sketch for higher-order parabolic flows of curves with boundary conditions is provided.  More details for the elastic flow, albeit with clamped boundary conditions, may also be found in \cite{Sp1, Sp2}.\\

In fact, $T=\infty$ is known for each of elastic-type flows considered here, for situations where no issues arise with the boundary.  In all cases the elastic energy is bounded under the flow.  For $\lambda \geq 0$ constant, one has by an inductive argument that $\int k^2 ds$ bounded implies all curvature derivatives are bounded in $L^2$, so it follows by a direct contradiction that $T=\infty$.  In the case $\lambda(t)$, an additional argument is needed using the Gagliardo-Nirenberg Sobolev inequality, Proposition \ref{p}.  We will state our main results with the assumption that neither end of the evolving curve reaches the cone tip (where the boundary conditions cease to make sense).  In the cases where we obtains exponential decay of the speed, it is straightforward to see that if the curve is initially sufficiently far from the cone tip, neither end can reach the cone tip under the evolution and thus $T=\infty$.  In the free elastic case the argument is similar, but instead we have the speed converging exponentially to that of an expanding circular arc, so if the ends are initially sufficiently far from the cone tip they will remain so.

An important ingredient in obtaining long time existence is the following bounds on the $L^2$ norms of all curvature derivatives under each of our flows.  These follow by similar arguments as in \cite{DKS02}.  We state the result as we will use it later in obtaining smooth exponential convergence of solutions to circular arcs (under rescaling in the case $\lambda=0$).  For the case of $\lambda(t)$ we will give the proof in Section \ref{S:lf}.\\

\begin{proposition}\label{bd}
Let $\alpha: \left[ -1, 1\right]\times  \left [ 0, T \right )\to \mathbb{R}^{2}$ is a solution to \eqref{f} and compatible with the generalised Neumann boundary conditions \eqref{E:NBC1} and \eqref{E:NBC2}. For each $\ell \in \mathbb{N}\cup \left\{0 \right\},$ there exists a constant $C_{\ell}^{2}> 0 $ such that $$\int _{\alpha }k_{s^\ell}^{2}\,ds\leq C_{\ell}^{2}$$ for all $t\in \left [ 0,\infty  \right )$.\\    
\end{proposition}

With the above preliminaries in place, we now discuss in the following sections the three cases of flow: $\lambda > 0,~\lambda=\lambda(t)$ and $\lambda =0.$

\section{Elastic flow with length penalty, $\lambda >0$ }\label{S:lp}
Here we add to the original elastic energy a length term, considering the family of energies
\begin{equation}\label{E}
E_{\lambda}\left [ \alpha_t  \right ]=E_0\left [ \alpha  \right ]+ \lambda~ L\left [ \alpha  \right ],    
\end{equation} where $E_0\left [ \alpha  \right ]=\frac{1}{2}\int _{\alpha }k^{2}ds$, and $\lambda > 0$ is a positive constant.  The corresponding $L^2$-gradient flow for this energy is 
\begin{equation}\label{eq11}
\frac{\partial \alpha }{\partial t}=(  k_{ss}+\frac{1}{2}k^{3}-\lambda~ k)~\nu \mbox{.}
\end{equation}

In our setting, for any fixed $\lambda>0$ there is a unique circular arc centred at the cone tip that is a stationary curve for \eqref{eq11} compatible with the boundary conditions.  Specifically, this circular arc, which we will denote by $\alpha_\infty$, has curvature satisfying
$$\frac{1}{2} k^3 - \lambda\, k = 0$$
and therefore
$$k = \sqrt{2\, \lambda} \mbox{.}$$

Since \eqref{eq11} is the $L^2$-gradient flow for \eqref{E} we immediately have the following:
\begin{lemma}
Under the flow \eqref{eq11}, the energy $ E_{\lambda}\left [ \alpha_t  \right ]$  is non-increasing. Thus for each $t$, 
$$E_{\lambda }\left [ \alpha_t  \right ]\leq E_{\lambda }\left [ \alpha_{0}  \right ]$$
where \begin{equation*}
\frac{\mathrm{d} }{\mathrm{d} t} E_{\lambda }\left [ \alpha_t  \right ]
 = -\int _{\alpha }\left<k_{ss}+\frac{1}{2}k^{3}-\lambda \, k,~F \right>ds =-\int _{\alpha }\left|F \right|^{2}ds \leq 0.\\
\end{equation*}    
\end{lemma}

Taking $F=- k_{ss} - \frac{1}{2} k^3 + \lambda \, k$ we obtain from Lemma \ref{T:evlneqns} the following:\\

\begin{lemma} \label{T:evlneqns1}
Under the flow \eqref{eq11} we have the following evolution equations for various pointwise geometric quantities.
\begin{itemize}
  \item[\textnormal{(i)}] $\displaystyle{\frac{\partial}{\partial t} ds = (  k_{ss}+\frac{1}{2}k^{3}-\lambda~ k)\, k\, ds}$;
  \item[\textnormal{(ii)}] For each $\ell= 0, 1, 2, \ldots$,
  $$\frac{\partial}{\partial t} k_{s^\ell} = -k_{s^{\ell+4}} - \sum_{j=0}^{\ell} \left[ k\, k_{s^{\ell-j}} (  k_{ss}+\frac{1}{2}k^{3}-\lambda~ k)\, \right]_{s^{j}} \mbox{.}$$ 
  \end{itemize}
  \end{lemma}
\mbox{}\\

 \begin{corollary}\label{T:evlneqns2}
  Under the flow \eqref{eq11}, 
   \begin{enumerate}
\item [(i)] $\displaystyle{\frac{\mathrm{d} }{\mathrm{d} t}L}=-\int _{\alpha }k_{s}^{2}ds+\frac{1}{2}\int _{\alpha }k^{4}ds-\lambda \int _{\alpha }k^{2}ds;$
\item [(ii)]$\displaystyle{\frac{\partial }{\partial t}k}=-k_{s^{4}}-3kk_{s}^{2}-\frac{5}{2}k^{2}k_{ss}-\frac{1}{2}k^{5}+\lambda k_{ss}+\lambda k^{3};$ 
\item [(iii)]$\displaystyle{\frac{\partial }{\partial t}}k_{s}=-k_{s^{5}}-3k_{s}^{3}-12kk_{s}k_{ss}-3k^{4}k_{s}+4\lambda k^{2}k_{s}-\frac{5}{2}k^{2}k_{s^{3}}+\lambda k_{s^{3}};$
\item [(iv)]$\displaystyle{\frac{\mathrm{d} }{\mathrm{d} t}\int_{\alpha } k^{2}ds}=-2\int_{\alpha }k_{s^{2}}^{2}ds+6\int_{\alpha }k^{2}k_{s}^{2}ds-\frac{1}{2}\int_{\alpha }k^{6}ds+\lambda \int_{\alpha }k^{4}ds+2\lambda \int_{\alpha }kk_{s^{2}}ds;$ 
\item [(v)]$\displaystyle{\frac{\mathrm{d} }{\mathrm{d} t}\int _{\alpha }k_{s}^{2}\,ds=-2\int _{\alpha }k_{s^{3}}^{2}\,ds+5\int _{\alpha }k_{ss}^{2}k^{2}\,ds-\frac{5}{3}\int _{\alpha }k_{s}^{4}\,ds-\frac{11}{2}\int _{\alpha }k_{s}^{2}k^{4}\,ds}$

\hspace*{\fill}$+7\lambda \int _{\alpha }k^{2}k_{s}^{2}\,ds-2\lambda \int _{\alpha }k_{ss}^{2}\,ds;$
\item [(vi)]$\displaystyle{\frac{\mathrm{d} }{\mathrm{d} t}\frac{1}{2}\int _{\alpha }k_{s^{l}}^{2}\,ds=-\int _{\alpha }k_{s^{l+2}}^{2}ds-\lambda \int _{\alpha }k_{s^{l+1}}^{2}ds+\lambda \int _{\alpha }P_{4}^{2l}(k)ds+\int _{\alpha }P_{4}^{2l+2}(k)ds}$

\hspace*{\fill}$+\int _{\alpha }P_{6}^{2l}(k)ds.$\\
\end{enumerate} 
\end{corollary} 

\begin{theorem} \label{T:main1}
Let $\lambda>0$ be constant.  Given $\alpha_0: \left[ -1, 1\right]\to \mathbb{R}^{2}$ compatible with \eqref{E:NBC1} and \eqref{E:NBC2}, then, under the assumption that neither end of the evolving curve reaches the cone tip, the length-penalised elastic flow \eqref{eq11} with generalised Neumann boundary conditions \eqref{E:NBC1} and \eqref{E:NBC2} has a unique solution $\alpha\left( \cdot, t\right)$ that exists for all time.  In the case that the supporting cone satisfies 
$$\omega \leq  \frac{1}{\sqrt{28}} \approx  0.19$$ 
and $\alpha_0$ satisfies a smallness condition of the form
$$\left\|k_{s} \right\|_{2}^2 \leq c\left( \omega, \underline{L}, \overline{L}, \lambda \right)\mbox{,}$$
then the solution curves converge smoothly and exponentially in the $C^\infty$-topology to the circular arc $\alpha_\infty$ of radius $\frac{1}{\sqrt{2\lambda } }$.\\
\end{theorem}

\begin{remarks}
\begin{enumerate}
    \item It is interesting to note that the boundary conditions and length penalty are enough to determine for any given initial curve $\alpha_0$ the unique limiting circular arc.  This is in contrast to the case of curves evolving between parallel lines \cite{WW} where the exact height of the limiting straight horizontal line cannot be determined.
    \item The smallness condition on $\left\| k_s \right\|_2$ referred to above is given explicitly by \eqref{E:smallness1}.\\
\end{enumerate}
\end{remarks}

\begin{proof}
As we have mentioned, short-term existence of a unique solution to the flow equation is essentially standard and well-known.  The fact that the solution to the flow exists for all time may be proven using a short contradiction argument.  Suppose that $T< \infty$.  Since $E_\lambda(\alpha_t)$ is non-increasing, we have
\begin{equation} \label{E:Ebound}
  \frac{1}{2} \int_{\alpha } k^2 ds + \lambda L \leq E_\lambda[\alpha _{0}]
\end{equation}
for any $t$ and at the assumed $T$.  Crucially, this gives that both $\int k^2 ds$ and $L$ remain bounded up until time $T$.

It follows using the evolution equation of Corollary \ref{T:evlneqns2}, (vi) and interpolation that $\int k_{s^\ell}^2\,ds$ is also bounded up until time $T$.  Thus, $\alpha\left( \cdot, T\right)$ is a smooth curve compatible with the generalised Neumann boundary conditions for which we may again apply short-time existence, contradicting the maximality of $T$.  We conclude $T=\infty$.\\
\end{proof}

Next, we observe similarly, as in \cite{lin2012l2} that the length of the evolving curve is also bounded below.  We estimate using the H\"{o}lder inequality and \eqref{E:Ebound}
$$\theta_1 - \theta_2 = 2\pi\omega = \int_{\alpha } k \, ds \leq L^{\frac{1}{2}} \left( \int k^2 ds\right)^{\frac{1}{2}} \leq L^{\frac{1}{2}} \left( 2 E_\lambda[\alpha _{0}]\right)^{\frac{1}{2}}$$
so
\begin{equation} \label{E:llower}
  L \geq \frac{\left( \theta_1 - \theta_2 \right)^2}{2 E_\lambda[\alpha _{0}]}=\frac{2\,\pi ^{2} \omega ^{2}}{E_\lambda[\alpha _{0}]} =: \underline{L}>0\mbox{.}\\
\end{equation}

We also have $L \leq \overline{L} := \frac{E_{\lambda }[\alpha _{0}]}{\lambda }$ from \eqref{E:Ebound}.
That $L \leq \overline{L}$ facilitates a much more direct argument for the geometric stability of circular arcs for small $\lambda>0$ as compared with the free elastic flow case.  We begin with the evolution equation for the key smallness quantity.\\

\begin{lemma}\label{1k}
Let $\alpha :\left [ -1, 1 \right ]\times \left [ 0,  \infty\right )\to \mathbb{R}^{2}$ be a length-penalised elastic flow. Then 
\begin{align*}
&\frac{\mathrm{d} }{\mathrm{d} t}\int _{\alpha }k_{s}^{2}\,ds\\
&\leq -\frac{1}{8}\int _{\alpha }k_{s^3}^{2}\,ds-\frac{13}{6}\bar{k}^{4}\int _{\alpha }k_{s}^{2}\,ds-2\lambda \int _{\alpha }k_{ss}^{2}\,ds+14\lambda\left ( \left\|k-\bar{k}^2 \right\|_{\infty }^{2}+\bar{k} \right )\int _{\alpha }k_{s}^{2}\,ds\\
&\quad-22\bar{k}^{3}\int _{\alpha }\left ( k-\bar{k} \right )k_{s}^{2}\,ds+5\int _{\alpha }\left ( k-\bar{k} \right )^{2}k_{ss}^{2}\,ds+10\bar{k}\int _{\alpha }\left ( k-\bar{k} \right )k_{ss}^{2}\,ds. 
\end{align*}
\end{lemma}
\begin{proof}
 As in \cite{TW24}, by using $k=(k-\bar{k})+\bar{k}$, we have 
\begin{align*}
5 \int k^2 k_{ss}^2 \, ds &= 5 \int_{\alpha } \left( k - \bar{k} \right)^2 k_{ss}^2 \, ds + 10 \bar{k} \int_{\alpha } \left( k - \bar{k} \right) k_{ss}^2 \, ds + 5 \bar{k}^2 \int_{\alpha } k_{ss}^2 \, ds, \\
-\frac{11}{2} \int_{\alpha } k^4 k_{s}^2 \, ds &= -\frac{11}{2} \int_{\alpha } \left( k - \bar{k} \right)^4 k_{s}^2 \, ds - 22 \bar{k} \int_{\alpha } \left( k - \bar{k} \right)^3 k_{s}^2 \, ds \\
&\quad - 33 \bar{k}^2 \int_{\alpha } \left( k - \bar{k} \right)^2 k_{s}^2 \, ds - 22 \bar{k}^3 \int_{\alpha } \left( k - \bar{k} \right) k_{s}^2 \, ds - \frac{11}{2} \bar{k}^4 \int_{\alpha } k_{s}^2 \, ds.
\end{align*}
Furthermore, using Young's inequality for any $b_0, b_1 >0,$ we obtain
\begin{align*}
5\bar{k}^2 \int_{\alpha } k_{ss}^2 \, ds &= -5\bar{k}^2 \int_{\alpha } k_{s^3} k_{s} \, ds \leq b_0 \int_{\alpha } k_{s^3}^2 \, ds + \frac{25}{4 b_0} k^4 \int_{\alpha } k_{s}^2 \, ds, \\
-22 \bar{k} \int_{\alpha } \left( k - \bar{k} \right)^3 k_{s}^2 \, ds &\leq 22 b_1 \bar{k}^2 \int_{\alpha } \left( k - \bar{k} \right)^2 k_{s}^2 \, ds + \frac{11}{2 b_1} \int_{\alpha } \left( k - \bar{k} \right)^4 k_{s}^2 \, ds.
\end{align*}
Also, we have 
$$7\lambda \int _{\alpha }k^{2}k_{s}^{2}\,ds\leq 14\lambda \left ( \left\|k-\bar{k} \right\|_{\infty }^{2}+\bar{k} \right )\int _{\alpha }k_{s}^{2}\,ds .$$ 
Using Lemma \ref{T:evlneqns2},(v), we have
\begin{align*}
\frac{\mathrm{d} }{\mathrm{d} t} \int_{\alpha } k_{s}^2 \, ds &\leq - (2 - b_0) \int_{\alpha } k_{s^3}^2 \, ds - \frac{5}{3} \int_{\alpha } k_{s}^4 \, ds - \frac{11}{2} \left( 1 - \frac{1}{b_1} \right) \int_{\alpha } \left( k - \bar{k} \right)^4 k_{s}^2 \, ds \\
&\quad - \left( 33 - 22 b_1 \right) \bar{k}^2 \int_{\alpha } \left( k - \bar{k} \right)^2 k_{s}^2 \, ds + \left( \frac{25}{4 b_0}  - \frac{11}{2}  \right)\bar{k}^4 \int_{\alpha } k_{s}^2 \, ds, \\
&\quad - 22 \bar{k}^3 \int_{\alpha } \left( k - \bar{k} \right) k_{s}^2 \, ds + 5  \int_{\alpha } \left( k - \bar{k} \right) k_{ss}^2 \, ds + 10 \bar{k} \int_{\alpha } \left( k - \bar{k} \right) k_{s}^2 \, ds\\&\quad +14\lambda \left ( \left\|k-\bar{k} \right\|_{\infty }^{2}+\bar{k}^2 \right )\int _{\alpha }k_{s}^{2}\,ds -2\lambda \int _{\alpha }k_{ss}^{2}\,ds.
\end{align*}
Let $b_0 = \frac{15}{8}$ and $b_1 = 1.$  Discarding some non-positive terms that are not useful completes the proof.
\end{proof}

\begin{lemma}\label{2k}
 Under the flow \eqref{eq11}, 
\begin{align}\label{E:ksl2}
&\frac{\mathrm{d} }{\mathrm{d} t}\int _{\alpha }k_{s}^{2}\,ds\\ \nonumber
&\leq -\bigg[ \frac{1}{8}-\left ( \frac{28\lambda \overline{L}^{2} }{\pi ^{2}}+10 \right )\frac{\overline{L}^{3}}{\pi ^{3}}\left\|k_{s} \right\|_{2}^{2}-\left ( 10(2\omega )+22(2\omega )^{3} \right )\sqrt{\frac{2\overline{L}^{3}}{\pi ^{3}}}\left\|k_{s} \right\|_{2}\bigg]\int _{\alpha }k_{s^{3}}^{2}\,ds\\ \nonumber
& \quad+\bigg[2\lambda\left ( 7\bar{k}^{2} -\frac{\pi ^{2}}{\underline{L}^{2}}\right )\bigg]\int _{\alpha }k_{s}^{2}\,ds-\frac{13}{6}\bar{k}^{4}\int _{\alpha }k_{s}^{2}\,ds. 
\end{align}
In particular, if $\omega \leq \frac{1}{\sqrt{28}}$ then
at any time $t$ such that
\begin{multline} \label{E:smallness1}
\left\|k_{s} \right\|_{2}^{2}\leq c\left( \omega, \underline{L}, \overline{L}, \lambda \right)\\
 := \frac{2\pi ^{7}}{(56\lambda \underline{L}^{2}+20\pi ^{2})^{2}}\left ( \sqrt{(176\omega^{3}+20\omega )^{2}+\frac{(14\lambda \overline{L}^{2}+5\pi ^{2})}{4\pi ^{2}}} -(176\omega^{3}+20\omega )\right )^{2},
\end{multline}
 we have
 \begin{equation}\label{ec}
\frac{\mathrm{d} }{\mathrm{d} t}\int _{\alpha }k_{s}^{2}\,ds\leq -\frac{13}{6}\bar{k}^4\int _{\alpha }k_{s}^{2}\,ds
\end{equation} 
and thus there are constants $C_1>0$, depending only on $\alpha_0$ and $\delta_1>0$, depending only on $\omega$ and $\underline{L}$ such that
\begin{equation}\label{exc1}
\left\| k_s \right\|_2^2 \leq C_1 e^{-\delta_1\, t} \mbox{.}    \end{equation}
 \end{lemma}
 \mbox{}\\
 
 \begin{remark}
By rewriting all terms on the second line of \eqref{E:ksl2} as a quadratic in $\omega^2$, it is possible to write down a smallness requirement on $\omega$ that depends on $\lambda$ and $E_\lambda\left[ \alpha_0 \right]$.  However, we have chosen here instead to simply give an absolute upper bound on $\omega$ by ensuring that the first term on the second line of \eqref{E:ksl2} is non-positive.     
 \end{remark}
 \begin{proof}
Applying Propositions \ref{psw} and \ref{eq266} similarly as in the corresponding argument in \cite{TW24}, we estimate
\begin{multline*}
14\lambda \left ( \left\|k-\bar{k} \right\|_{\infty }^{2}+\bar{k} \right )\int _{\alpha }k_{s}^{2}\,ds  \leq 28 \lambda \frac{L}{\pi }\int _{\alpha }k_{s}^{2}\,ds\int _{\alpha }k_{s}^{2}\,ds+14\lambda \bar{k}^{2}\int _{\alpha }k_{s}^{2}\,ds \\
 \leq  28 \lambda \frac{L^5}{\pi^5 }\int _{\alpha }k_{s^3}^{2}\,ds\int _{\alpha }k_{s}^{2}\,ds+14\lambda \bar{k}^{2}\int _{\alpha }k_{s}^{2}\,ds
\end{multline*}
and 
$$-2\lambda \int _{\alpha }k_{ss}^{2}\,ds\leq -2\lambda \frac{\pi ^{2}}{L^{2}}\int _{\alpha }k_{s}^{2}\,ds,$$
\begin{equation*}
-22\bar{k}^{3}\int _{\alpha }\left ( k-\bar{k} \right )k_{s}^{2}\,ds \leq 22\left ( \frac{2\omega \pi }{L} \right )^{3}\left\|k-\bar{k} \right\|_{\infty }\left\| k_{s}\right\|_{2}^{2} 
 \leq 22\left ( 2\omega  \right )^{3} \sqrt{\frac{2L^{3}}{\pi ^{3}}}\left\| k_{s}\right\|_{2}\left\| k_{s^{3}}\right\|_{2}^{2}
\end{equation*}
$$5\int _{\alpha }\left ( k-\bar{k} \right )^{2}k_{ss}^{2}\,ds\leq 5\left\|k-\bar{k} \right\|_{\infty }^{2}\left\|k_{ss} \right\|_{2}^{2}\leq \frac{10L^{3}}{\pi ^{3}}\left\|k_{s} \right\|_{2}^{2}\left\|k_{s^3} \right\|_{2}^{2},$$
\begin{equation*}
10\bar{k}\int _{\alpha }\left ( k-\bar{k} \right )k_{s}^{2}\,ds \leq 22\left ( \frac{2\omega \pi }{L} \right )\left\|k-\bar{k} \right\|_{\infty }\left\| k_{ss}\right\|_{2}^{2} 
  \leq 10\left ( 2\omega  \right ) \sqrt{\frac{2L^{3}}{\pi ^{3}}}\left\| k_{s}\right\|_{2}\left\| k_{s^{3}}\right\|_{2}^{2}.
\end{equation*}
By inserting these estimates into the inequality of Lemma \ref{1k}, we obtain
\begin{align*}
&\frac{\mathrm{d} }{\mathrm{d} t}\int _{\alpha }k_{s}^{2}\,ds\\
&\leq -\bigg[ \frac{1}{8}-\left ( \frac{28\lambda L^{2} }{\pi ^{2}}+10 \right )\frac{L^{3}}{\pi ^{3}}\left\|k_{s} \right\|_{2}^{2}-\left ( 10(2\omega )+22(2\omega )^{3} \right )\sqrt{\frac{2L^{3}}{\pi ^{3}}}\left\|k_{s} \right\|_{2}\bigg]\int _{\alpha }k_{s^{3}}^{2}\,ds\\&\quad +\bigg[\left ( 7\bar{k}^{2} -\frac{\pi ^{2}}{L^{2}}\right )\bigg]\int _{\alpha }k_{s}^{2}\,ds-\frac{13}{6}\bar{k}^{4}\int _{\alpha }k_{s}^{2}\,ds\\&\leq -\bigg[ \frac{1}{8}-\left ( \frac{28\lambda L^{2} }{\pi ^{2}}+10 \right )\frac{L^{3}}{\pi ^{3}}\left\|k_{s} \right\|_{2}^{2}-\left ( 10(2\omega )+22(2\omega )^{3} \right )\sqrt{\frac{2L^{3}}{\pi ^{3}}}\left\|k_{s} \right\|_{2}\bigg]\int _{\alpha }k_{s^{3}}^{2}\,ds\\&\quad +\bigg[\frac{2\lambda \pi ^{2}}{L^{2}}\left ( 7(2\omega )^{2}-1 \right )\bigg]\int _{\alpha }k_{s}^{2}\,ds-\frac{13}{6}\bar{k}^{4}\int _{\alpha }k_{s}^{2}\,ds
\end{align*}
 and if $\omega \leq \frac{1}{\sqrt{28}}$
 we have
 \begin{align*}
&\frac{\mathrm{d} }{\mathrm{d} t}\int _{\alpha }k_{s}^{2}\,ds\\
&\leq -\bigg[ \frac{1}{8}-\left ( \frac{28\lambda L^{2} }{\pi ^{2}}+10 \right )\frac{L^{3}}{\pi ^{3}}\left\|k_{s} \right\|_{2}^{2}-\left ( 10(2\omega )+22(2\omega )^{3} \right )\sqrt{\frac{2L^{3}}{\pi ^{3}}}\left\|k_{s} \right\|_{2}\bigg]\int _{\alpha }k_{s^{3}}^{2}\,ds\\
&\quad -\frac{13}{6}\bar{k}^{4}\int _{\alpha }k_{s}^{2}\,ds.
\end{align*}
Finding the smallest positive root $x$ of $ \frac{1}{8}-\left ( \frac{28\lambda \overline{L}^{2} }{\pi ^{2}}+10 \right )\frac{\overline{L}^{3}}{\pi ^{3}}x^2-\left ( 10(2\omega )+22(2\omega )^{3} \right )\sqrt{\frac{2L^{3}}{\pi ^{3}}}x$, where $x^2=\left\|k_{s} \right\|_{2}^{2},$ we obtain that if 
\begin{equation}\label{x_0}
\left\|k_{s} \right\|_{2}\leq \frac{\pi ^{5}}{(56\lambda \underline{L}^{2}+20\pi ^{2})}\sqrt{\frac{2}{\pi ^{3}}}\left[ \sqrt{(176\omega^{3}+20\omega )+\frac{(14\lambda \overline{L}^{2}+5\pi ^{2})}{4\pi ^{2}}} -(176\omega^{3}+20\omega )\right],\end{equation}
then the required inequality is true.  The stated exponential decay follows.\\  
 \end{proof}

\begin{corollary} \label{C:expdecay}
Let $\alpha :\left [ -1, 1 \right ]\times \left [ 0,  \infty\right )\to \mathbb{R}^{2}$ be a length-penalised elastic flow with conditions as per Theorem \ref{T:main1}. Then for all $\ell \in \mathbb{N}$, there exist corresponding $\delta_\ell >0$ and $\tilde C_\ell >0$ such that
  \begin{equation}
 \int _{\alpha }k_{s^\ell}^2\,ds\leq \tilde C_\ell  \, e^{-\delta_\ell t} \mbox{.}  
  \end{equation} 
\end{corollary}
\begin{proof}
This is proved by a standard induction argument.  The base case $\ell=1$ is as in  \eqref{exc1}.  So assume that
 $$\int _{\alpha }k_{s^\ell}^{2}\,ds\leq \tilde C_\ell e^{-\delta_\ell t}$$ for some $\tilde{C_\ell}>0$.  Integrating by parts and using the property at the boundaries, Corollary \ref{T:BCs}, we estimate using the H\"{o}lder inequality
$$\int _{\alpha }k_{s^{\ell+1}}^2\,ds=-\int _{\alpha }k_{s^{\ell+2}}\,k_{s^{\ell}}\,ds\leq \left ( \int _{\alpha }k_{s^{\ell+2}}^2\,ds \right )^{\frac{1}{2}}\left ( \int _{\alpha }k_{s^{\ell}}^2\,ds \right )^{\frac{1}{2}}$$
 The inductive assumption, Proposition \ref{bd} and Lemma \ref{2k}  yields $$\int _{\alpha }k_{s^{\ell+1}}^2\,ds\leq C_{\ell+2}\, \tilde{C_\ell}^\frac{1}{2}e^{-\frac{\delta_\ell}{2} t}$$
 Taking $\tilde C_{\ell+1} = C_{\ell+2} \tilde C_\ell^{\frac{1}{2}}$ and $\delta_{\ell+1} = \frac{\delta_\ell}{2}$ completes the proof. 
\end{proof}

\begin{remark} As usual, in the above induction argument, the constants $\delta_\ell$ are getting progressively smaller.  To obtain exponential decay of all curvature derivatives in $L^2$ with the same constant, one can use a standard linearisation argument about the limiting curve.  One then has all curvature derivatives decaying pointwise again using Proposition \ref{psw} and the length bound.\\
\end{remark}

Let us also note using Proposition \ref{psw} with Lemma \ref{2k} we obtain
$$\int _{\alpha }\left ( k-\bar{k} \right )^{2}\,ds\leq \frac{L^{2}}{\pi ^{2}}\int _{\alpha }k_{s}^{2}\,ds\leq \frac{\overline{L}^{2}}{\pi ^{2}}\int _{\alpha }k_{s}^{2}\,ds\leq c\,e^{-\delta t}$$ and,
$$\left\|k-\bar{k} \right\|_{\infty }^{2}\leq \frac{2L}{\pi }\int _{\alpha }k_{s}^{2}\,ds\leq \frac{2\overline{L}}{\pi }\int _{\alpha }k_{s}^{2}\,ds\leq c\,e^{-\delta t}$$ 

\begin{proof}[Completion of the proof of Theorem~{\upshape\ref{T:main1}}]
Given that $T= \infty$ and the curvature $k$ decays pointwise exponentially to its average, we know precisely by virtue of the boundary conditions the limiting circular arc.  Smooth exponential convergence of the embedding map $\alpha\left( \cdot, t\right)$ to that of the circular arc now follows similarly as in \cite{WW}, also as in \cite{GM24}.  The change to the latter argument is that we do not have constant area bounded by the curve and the sides of the cone in this setting.  Instead, we have since $T=\infty$ and the energy is monotone decreasing and bounded below by zero, at least a subsequence of times with convergence to a critical point of the energy.  We also have in the limit $k_{ss}=0$ so this critical point is exactly the circular arc with limiting radius $r_\infty = \frac{1}{\sqrt{2\lambda}}$.  Smooth exponential convergence at the level of curvature then follows by our previous estimates.  Control of the tangent vector and then spatial parameter derivatives and mixed derivatives of the embedding map then follow as in \cite{GM24} for example.  We conclude exponential convergence in the $C^\infty$-topology to the circular arc. This completes the proof.
\end{proof}
\section{Length-constrained elastic flow $\lambda (t)$} \label{S:lf}
In this section, we define an energy with $\lambda$ chosen as a function of time that ensures the length of the evolving curve is fixed.  In other words, the gradient flow of
\begin{equation}\label{E2}
E_0\left [ \alpha_t  \right ]=\frac{1}{2}\int _{\alpha }k^{2}ds,    
\end{equation} 
subject to fixed length $L\left[ \alpha_t\right] \equiv L_0$.  The corresponding flow equation is 
\begin{equation}\label{eq113}
\frac{\partial \alpha }{\partial t}=(  k_{ss}+\frac{1}{2}k^{3}-\lambda(t) k)\nu \mbox{,}
\end{equation}
and \begin{equation}\label{h1}
\lambda\left ( t \right ):=\displaystyle{\frac{-\int _{\alpha }k_{s}^{2}ds+\frac{1}{2}\int _{\alpha }k^{4}ds}{\int _{\alpha }k^{2}ds}}    
\end{equation}
ensures $L[\alpha_t] = L[\alpha_0] =: L_0$ for all $t$, as can be checked by direct calculation (Lemma \ref{T:evlneqns48}).\\

Again, short-time existence of a solution to the flow equation \eqref{eq113} is well known, and 
 subconvergence to an elastica follows in the usual way (see, for example, \cite{DLP17}, notwithstanding the slightly different boundary conditions).\\
\begin{lemma} \label{T:evlneqns48}
Under the flow \eqref{eq113}, while a solution exists, the length is constant  
$L\left [ \alpha_{t}  \right ]=L_{0}.$
\end{lemma}
\begin{proof}
By differentiating the length in time, we get
$$\frac{\mathrm{d} }{\mathrm{d} t}L\left [ \alpha_{t}  \right ]=- \int_{\alpha }kF~ds=-\int_{\alpha }k_{s}^{2}ds+\frac{1}{2}\int_{\alpha }k^{4}ds-\lambda (t)\int_{\alpha }k^{2}~ds = 0,$$ 
where in the last step we have used \eqref{h1}.  The result follows.      
\end{proof}
As shown in \cite{DLP17}, the flow \eqref{eq113} has the following important property.\\

\begin{lemma} \label{T:EE}
Under the flow \eqref{eq113}, the original elastic energy $E_0\left[ \alpha_t\right] =\frac{1}{2} \int_\alpha k^2 ds$ is non-increasing.\\
\end{lemma}

In view of Lemma \ref{T:EE} and the length constraint, we have that solutions to this flow also exist for all time, $T=\infty$.

For the sake of completeness, we now provide the additional argument needed to see that all curvature derivatives are bounded in $L^2$ under the flow \eqref{eq113}.
\begin{proof}[Proof of Proposition~{\upshape\ref{bd}}, $\lambda(t)$ case]
Using Lemma \ref{T:evlneqns2} (vi), we obtain 
\begin{align}\label{b332}
\frac{\mathrm{d} }{\mathrm{d} t}\frac{1}{2}\int _{\alpha }k_{s^{l}}^{2}\,ds\nonumber&=-\int _{\alpha }k_{s^{l+2}}^{2}\,ds-\lambda \left ( t \right ) \int _{\alpha }k_{s^{l+1}}^{2}\,ds+\lambda \left ( t \right ) \int _{\alpha }P_{4}^{2l}(k)\,ds+\int _{\alpha }P_{4}^{2l+2}(k)\,ds\nonumber\\ &\quad+\int _{\alpha }P_{6}^{2l}(k)\,ds\nonumber\\&\leq -(1-3\varepsilon)\int _{\alpha }k_{s^{l+2}}^{2}\,ds-\lambda \left ( t \right ) \int _{\alpha }k_{s^{l+1}}^{2}\,ds+\lambda \left ( t \right ) \int _{\alpha }P_{4}^{2l}(k)\,ds\nonumber\\&\quad + c \varepsilon ^{-(2\ell+3)}\left (\int _{\alpha }k^{2}\,ds  \right )^{2\ell+5}+c\left (\int _{\alpha }k^{2}\,ds  \right )^{2\ell+5}\nonumber\\&\quad +c \varepsilon ^{-(\ell+1)}\left (\int _{\alpha }k^{2}\,ds  \right )^{2\ell+5}+c\left (\int _{\alpha }k^{2}\,ds  \right )^{2\ell+5}.
\end{align} 
As in the proof of \cite{DKS02}[Theorem 3.3] we may estimate using Proposition \ref{p}, Young's inequality and the H\"{o}lder inequality that
$$\left| \lambda(t)\right| \leq c_\ell\left( \Lambda\right) \left( \left\|k_{s^{\ell+2}} \right\|_{2}^{\frac{2}{\ell+2}} + 1\right)$$
and so
$$\lambda(t) \int_\alpha P_4^{2\ell}\left( k \right) ds \leq \varepsilon \int_\alpha k_{s^{\ell+2}}^2 ds + c_\ell\left( \Lambda, \varepsilon\right) \mbox{.}$$

For the term $-\lambda \left ( t \right )\int _{\alpha }k_{s^{\ell+1}}^{2}\,ds$ we may proceed as in \cite{DKS02}, noting that the method there will work with minimal adjustments given our boundary conditions, as compared with other boundary conditions where different approaches are required (eg \cite{DLP14}).  Using Lemma \ref{T:divform} we can rewrite equation \eqref{eq113} as
$$\frac{\partial \alpha}{\partial t} = + \partial_s\left( k_s \nu + \frac{1}{2} k^2 \tau \right) - \lambda(t) \, k\, \nu \mbox{.}$$
Taking the inner product with $\alpha$ we consider
\begin{equation} \label{E:IP}
  \left< \frac{\partial \alpha}{\partial t}, \alpha \right> = \left< \partial_s \left( k_s \nu - \frac{1}{2} k^2 \tau \right), \alpha \right> + \lambda(t) \left< \kappa, \alpha \right> \mbox{.}
\end{equation}
Observe by the Fundamental Theorem of Calculus
$$\int_\alpha \partial_s \left< k_s \nu - \frac{1}{2} k^2 \tau, \alpha \right> ds = \left[ \left< k_s \nu - \frac{1}{2} k^2 \tau, \alpha \right> \right]_{\partial \alpha} \mbox{.}$$
The right hand side is equal to zero by the no-curvature-flux and Neumann boundary conditions, \eqref{E:NBC1} and \eqref{E:NBC2}.  Hence we have the `integration by parts'
$$\int_\alpha \left< \partial_s \left( k_s \nu - \frac{1}{2} k^2 \tau \right), \alpha \right> ds = - \int_\alpha \left<  k_s \nu - \frac{1}{2} k^2 \tau, \tau \right> ds =  \frac{1}{2} \int_\alpha k^2 ds$$
and the integration of \eqref{E:IP} yields with further simple estimates
$$E_0\left[ \alpha_t \right] - \lambda \, L_0 = \int_\alpha \left< \frac{\partial \alpha}{\partial t}, \alpha \right> ds \leq \int_\alpha \left| \frac{\partial \alpha}{\partial t} \right| \left| \alpha \right| ds \leq c(\omega) \left\| \frac{\partial \alpha}{\partial t} \right\|_{L^2} L_0^{\frac{3}{2}}\mbox{.}$$
Note that using elementary trigonometry, the above $c(\omega)$ may be taken explicitly as $\frac{1}{2} \tan\left( \pi \omega\right)$.
It follows that
$$-\lambda(t) \leq c\left( \omega\right) L_0^{\frac{1}{2}} \left\| \frac{\partial \alpha}{\partial t} \right\|_{L^2}$$
from which the argument now follows as in \cite{DKS02} that
$$\frac{\mathrm{d} }{\mathrm{d} t} \int_\alpha k_{s^\ell}^2 ds + c_0 \int_\alpha k_{s^\ell}^2 ds \leq c_\ell(\Lambda) \left[ 1 + \lambda^-(t)^2 \int_\alpha k_{s^\ell}^2 ds \right]$$
where $\lambda^-(t) = - \min\left\{ \lambda(t), 0 \right\}$.  The bound on $\int_\alpha k_{s^\ell}^2 ds$ then follows using Gr\"{o}nwall's inequality.
\end{proof}

We have the following main result for this section.\\

 \begin{theorem} \label{T:main2}
Let $\alpha_0: \left[ -1, 1\right]\to \mathbb{R}^{2}$ be compatible with \eqref{E:NBC1} and \eqref{E:NBC2} be a given initial curve of length $L_0$.  Under the assumption that neither end of the evolving curve reaches the cone tip, the length-constrained elastic flow \eqref{eq113} with generalised Neumann boundary conditions \eqref{E:NBC1} and \eqref{E:NBC2} has a unique solution $\alpha\left( \cdot, t\right)$ that exists for all time.  In the case that the supporting cone satisfies
\begin{equation} \label{E:o2}
  \omega < \sqrt[4]{\frac{15}{6592}} \approx 0.22
  \end{equation}
and $\alpha_0$ satisfies a smallness condition of the form
$$\left\| k_s \right\|_2 < c\left( \omega, L_0 \right),$$
then the solution curves converge smoothly and exponentially in the $C^\infty$-topology to the circular arc $\alpha_\infty$ of length $L_0$ centred at the cone tip.\\
\end{theorem}

\begin{remarks}
\begin{enumerate}[1.]
\item It is easy to check that any circular arc centred at the cone tip is stationary under \eqref{eq113}.  Given that the length of any evolving curve is constant under the flow \eqref{eq113}, the particular circular arc to which a solution converges must be the one with the same length $L_0$ as that of the initial curve $\alpha_0$.  This unique limiting circular arc for given $\alpha_0$ has radius $r_{0}=\frac{L_{0}}{2\pi \omega }.$
\item  The restriction on $\omega$ amounts to a cone angle of not more than about $79^o$.
\item The smallness condition $c\left( \omega, L_0 \right)$ on $\left\| k_s\right\|_2$ is precisely the smallest positive root of the quartic 
\begin{multline} \label{E:scd}
 \delta -\frac{28 L_{0}^{6}}{\pi ^{6}}\left\|k_{s} \right\|_{2}^{4}-14\, (2)^\frac{5}{2}\,\omega \left ( \frac{L_0}{\pi } \right )^\frac{9}{2}\left\|k_{s} \right\|_{2}^{3}-\left ( 10+\frac{2}{\pi (2\omega )^2} +22(2\omega )^{2}\right )\frac{L_{0}^{3}}{\pi ^{3}}\left\| k_{s}\right\|_{2}^{2}\\
 -\left (\left ( 22(2\omega )^{3}+10(2\omega ) \right )\sqrt{\frac{2L_{0}^{3}}{\pi ^{3}}}+\frac{14(2\omega )^{3}L_0}{\pi }\right )\left\| k_{s}\right\|_{2} 
\end{multline}
where
$$\delta = \frac{15}{8} - \frac{103}{2} (2\omega)^4 >0$$
since $\omega$ satisfies \eqref{E:o2}.\\
\end{enumerate}    
\end{remarks}

We begin by bounding $\lambda(t)$ above and below in terms of $\int_\alpha k_s^2 ds$, the quantity that we will soon show, under conditions of Theorem \ref{T:main2}, remains small.\\

\begin{lemma} \label{T:lambdabounds}
Under the flow \eqref{eq113},
$$-\frac{L_0 \int_\alpha k_s^2 ds}{\left( 2 \pi\omega\right)^2} \leq \lambda(t) \leq \frac{2\, L_0}{\pi} \int_\alpha k_s^2 ds + \overline{k}^2 \mbox{.}$$
\end{lemma}
\begin{proof}
We estimate 
\begin{equation*}
\int_{\alpha } k^4 \, ds = \int_{\alpha } \left[ \left( k - \bar{k} \right)^2 + 2 \bar{k} \left( k - \bar{k} \right) + \bar{k}^2 \right] k^2 \, ds \leq 2 \left( \left\| k - \bar{k} \right\|_{\infty}^2 + \bar{k}^2 \right) \int_{\alpha } k^2 \, ds
\end{equation*}
and so, using also Proposition \ref{eq266},
$$\lambda(t) \leq \frac{1}{2} \frac{\int_{\alpha } k^4 \, ds}{\int_{\alpha } k^2 \, ds} \leq \left\| k - \bar{k} \right\|_{\infty}^2 + \bar{k}^2 \leq \frac{{ 2 L_0}}{\pi} \int_{\alpha } k_s^2 \, ds + \bar{k}^2.$$ 
Now, since $$\int_{\alpha } k^2 \, ds\geq\frac{\left ( 2\pi \omega  \right )^{2}}{L_0}$$ 
we have $$\lambda \left ( t \right )\geq -\frac{\int_{\alpha } k^2_s \, ds}{\int_{\alpha } k^2 \, ds}\geq -\frac{L_0\,\int_{\alpha } k^2_s \, ds}{\left ( 2\pi \omega  \right )^{2}}.$$
This completes the proof.    
\end{proof}

Similarly, as in the previous section, we have the following evolution equation, where now $\lambda = \lambda(t)$.\\

\begin{lemma} \label{T:ksL2}
Under the flow \eqref{eq113},
 \begin{multline*}\frac{\mathrm{d} }{\mathrm{d} t}\int _{\alpha }k_{s}^{2}ds=-2\int _{\alpha }k_{s^{3}}^{2}\,ds+5\int _{\alpha }k_{ss}^{2}k^{2}\,ds-\frac{5}{3}\int _{\alpha }k_{s}^{4}\,ds\\
 -\frac{11}{2}\int _{\alpha }k_{s}^{2}k^{4}\,ds+7\lambda\left ( t \right ) \int _{\alpha }k^{2}k_{s}^{2}\,ds -2\lambda \left ( t \right )\int _{\alpha }k_{ss}^{2}\,ds.
\end{multline*}  
\end{lemma}

\begin{corollary}\label{sk}
Under the flow \eqref{eq113},
\begin{align*}
&\frac{\mathrm{d} }{\mathrm{d} t}\int _{\alpha }k_{s}^{2}\,ds\\
&\leq- \bigg[\frac{15}{8} -\frac{28 L_{0}^{6}}{\pi ^{6}}\left\|k_{s} \right\|_{2}^{4}-14\, (2)^\frac{5}{2}\,\omega \left ( \frac{L_0}{\pi } \right )^\frac{9}{2}\left\|k_{s} \right\|_{2}^{3} -\left ( 10+\frac{2}{\pi (2\omega )^2} +22(2\omega )^{2}\right )\frac{L_{0}^{3}}{\pi ^{3}}\left\| k_{s}\right\|_{2}^{2}\\
 & \qquad \quad -\left (\left ( 22(2\omega )^{3}+10(2\omega ) \right )\sqrt{\frac{2L_{0}^{3}}{\pi ^{3}}}+\frac{14(2\omega )^{3}L_0}{\pi }\right )\left\| k_{s}\right\|_{2} -\frac{103}{2}(2\omega )^{4}\bigg]\left\|k_{s^{3}} \right\|_{2}^{2}\\
\end{align*}\end{corollary}
\begin{proof}
Unlike the $\lambda > 0$ constant case, we have to estimate the last term on the right hand side in Lemma \ref{T:ksL2} as well, but this is straightforward.  We have using Lemma \ref{T:lambdabounds} and Propositions \ref{psw} and \ref{eq266},
$$-2 \lambda(t) \int_{\alpha } k_{ss}^2 \, ds \leq \frac{2 L_0^{3}}{\pi^2 (2 \pi \omega)^{2}} \| k_s \|_2^2 \int_{\alpha } k_{s^3}^2 \, ds,$$
and 
\begin{align*}
&7 \lambda(t) \int_{\alpha } k^2 k_s^2 \, ds\\
&\leq 7 \left[ \frac{ 2\,L_0}{\pi} \int_{\alpha } k_s^2 \, ds + \bar{k}^2 \right] \int_{\alpha } k^2 k_s^2 \, ds\\
 &\leq 7\left[ \frac{ 2\,L_0}{\pi} \int_{\alpha } k_s^2 \, ds + \bar{k}^2 \right] \left [ \int _{\alpha }\left ( k-\bar{k} \right )^{2}\,k_{s}^{2}\,ds+2\bar{k}\int _{\alpha }\left ( k-\bar{k} \right )\,k_{s}^{2}\,ds+\bar{k}^{2}\int _{\alpha }k_{s}^{2}\,ds \right ] \\
 &\leq 7\left[ \frac{ 2\,L_0}{\pi} \int_{\alpha } k_s^2 \, ds + \bar{k}^2 \right]\left [ \left\|k-\bar{k} \right\|_{\infty }^{2}+2\bar{k}\left\|k-\bar{k} \right\|_{\infty } +\bar{k}^{2} \right ] \int _{\alpha }k_{s}^{2}\,ds\\
 &\leq 7\left[ \frac{ 2\,L_0}{\pi} \int_{\alpha } k_s^2 \, ds + \bar{k}^2 \right] \left[ \frac{2L_{0}}{\pi }\int _{\alpha }k_{s}^{2} \,ds +2\bar{k}\sqrt{\frac{2L_0}{\pi }}\left\|k_{s} \right\|_{2} +\bar{k}^{2} \right ]  \int _{\alpha }k_{s}^{2}\,ds \\
 &\leq \left[ \frac{28L_{0}^{6}}{\pi ^{6}}\left\| k_{s}\right\|_{2}^{4}+14(2)^{\frac{3}{2}}\bar{k}\left ( \frac{L_{0}}{\pi } \right )^{\frac{11}{2}} \left\|k_{s} \right\|_{2}^{3}+\frac{42L_{0}^{5}}{\pi^{5} }\bar{k}^{2} \left\| k_{s}\right\|_{2}^{2} +\frac{14L^4_0}{\pi^4 }\bar{k}^{3}\left\|k_{s} \right\|_{2} \right] \left\| k_{s^{3}}\right\|_{2}^{2}\\
 &\quad +7\bar{k}^{4}\left\| k_{s}\right\|_{2}^{2}.
\end{align*}

Estimating similarly as in the proof of Lemma \ref{1k}, except with $b_0= \frac{1}{8}$, we obtain
\begin{align*}
&\frac{\mathrm{d} }{\mathrm{d} t} \int_{\alpha } k_s^2 \, ds\\
&\leq -\frac{15}{8} \int_{\alpha } k_{s^3}^2 \, ds +\frac{89}{2} \bar{k}^4 \int_{\alpha } k_s^2 \, ds -\frac{5}{3}\int _{\alpha }k_{s}^{4}\,ds - 22 \bar{k}^3 \int_{\alpha } (k - \bar{k}) k_s^2 \, ds   \\
 &\quad + 5 \int_{\alpha } (k - \bar{k})^2 k_{ss}^2 \, ds + 10 \bar{k} \int_{\alpha } (k - \bar{k}) k_{ss}^2 \, ds + \frac{2 L_0^{3}}{\pi^2 (2 \pi \omega)^{2}} \| k_{s} \|_2^2 \int_{\alpha } k_{s^3}^2 \, ds\\
 &\quad+ \left[ \frac{28L_{0}^{6}}{\pi ^{6}}\left\| k_{s}\right\|_{2}^{4}+14(2)^{\frac{3}{2}}\bar{k}\left ( \frac{L_{0}}{\pi } \right )^{\frac{11}{2}} \left\|k_{s} \right\|_{2}^{3}+\frac{42L_{0}^{5}}{\pi^{5} }\bar{k}^{2} \left\| k_{s}\right\|_{2}^{2} +\frac{14L^4_0}{\pi^4 }\bar{k}^{3}\left\|k_{s} \right\|_{2} \right] \left\| k_{s^{3}}\right\|_{2}^{2}\\
 & \quad +\frac{14L^4_0}{\pi^4 }\bar{k}^{3}\left\|k_{s} \right\|_{2}\left\|k_{s^3} \right\|_{2}^{2} +7\bar{k}^{4}\left\| k_{s}\right\|_{2}^{2} -11 \bar{k}^{2}\int _{\alpha }\left ( k-\bar{k} \right )^{2}k_{s}^{2}\,ds.
\end{align*}

Working now with the remaining terms as in the proof of Lemma \ref{2k} we obtain the desired result.\\ 
\end{proof}

\begin{proposition}\label{exksc}
 Let $\alpha :\left [ -1, 1 \right ]\times \left [ 0,  \infty\right )\to \mathbb{R}^{2}$ be a length-constrained elastic flow with generalised Neumann boundary conditions \eqref{E:NBC1} and \eqref{E:NBC2}, supported in a cone satisfying \eqref{E:o2} and $\alpha_0$ satisfying the smallness condition 
 $$ \left\| k_s\right\|_2\leq c\left( \omega, L_0 \right)$$ 
 where $c\left( \omega, L_0 \right)$ is as per \eqref{E:scd}. Then there exist constants $C>0$ and $\delta>0$ such that 
  \begin{equation}
 \int _{\alpha }k_{s}^2\,ds\leq C  e^{-\delta t} ~~~~~~~~~\textit{for all}~~~t\geq0
  \end{equation} 
that is $\int _{\alpha }k_{s}^{2}\,ds$ decaying exponentially to zero. \end{proposition}
\begin{proof}
All terms in Corollary \ref{sk} may be handled using smallness, except that we need
$$\frac{103}{2}\times 2^{4} \omega ^{4}< \frac{15}{8}-\delta $$
for some $\delta > 0.$ Under this assumption, corresponding to \eqref{E:o2}, we then have for sufficiently small $\left\| k_s\right\|_2$ that
$$\frac{d}{dt} \int k_s^2 ds \leq - \delta \int k_s^2 ds$$
from which the result follows.\\   
\end{proof}

\begin{proof}[Completion of the proof of Theorem~{\upshape\ref{T:main2}}]
 The argument is similar as in the previous section.  Since $T=\infty$ and the curvature decays pointwise exponentially to its average, the target limit is the unique circular arc satisfying the boundary conditions and with length $L_0$.  In view of the exponential decay of curvature to its average, there is a time beyond which the curve remains convex.  Moreover, in view of pointwise exponential decay of all curvature derivatives, we can always find a larger time beyond which the curve is $C^{4, \alpha}$-close to the limiting circular arc. A stability argument, similar to that outlined in \cite{MWY}, then provides exponential convergence to the limiting circular arc that, in this setting is unique in view of the fixed length and the boundary conditions.
\end{proof}

\section{Free elastic flow, $\lambda =0$ \label{S:f}}

Some of our work for $\lambda>0$ carries over in this case; however, critically, observe that the length upper bound no longer holds.  This is not surprising if we consider an initial circular arc (centred at the cone tip) of radius $r_0$.  Under the flow \eqref{eq11} with $\lambda =0$, the circular arc expands self-similarly and indefinitely according to 
$$\frac{dr}{dt} = \frac{1}{2} r^{-3} \mbox{,}$$
that is,
$$r(t) = \sqrt[4]{r_0^4 + 2 t} \mbox{.}$$
By analogy with our earlier results, it is reasonable to conjecture that in the case $\lambda=0$, any smooth initial curve with $\int_\alpha k_s^2\, ds$ sufficiently small will give rise to a flow whose solution exists for all time and converges exponentially to an expanding circular arc, at least for small cone angle.  In this section, we prove that this is indeed the case.  It turns out, in contrast to the length-penalised and length-constrained elastic flows, no smallness condition is required on the cone angle in this case.\\

\begin{theorem} \label{T:main3}
Let $\alpha\left ( \cdot ,0 \right )=\alpha _{0}$ be a given initial curve, compatible with the boundary conditions \eqref{E:NBC1} and \eqref{E:NBC2}  and satisfying 
\begin{equation}
\varepsilon(0)= L^{3}\left [ \alpha _{0} \right ]\int _{\alpha _{0}}k_{s}^{2}\,ds< \varepsilon_0.   
\end{equation}
 for some $\varepsilon_0=\varepsilon_0(\omega).$  Provided neither end of the evolving curve reaches the cone tip, the flow \eqref{eq11}, with $\lambda =0$, has a solution for all time. The curves $\alpha\left ( \cdot ,t \right )$ are smooth and converge smoothly and exponentially in the $C^\infty$-topology to a self-similar expanding circular arc $\alpha_{\infty}$, with radius $r(t) = \sqrt[4]{r_0^4 + 2 t}$, where  $r_0$ is given via $L[\alpha _{0}]=2\pi\omega r_{0}.$\\
\end{theorem}

\begin{remark}
\begin{enumerate}
  \item Our theorem is the analogue of \cite{TW24}[Theorem 1.2] which is for closed, immersed curves.  The strategy of our proof is similar, so we provide only an outline, pointing out where differences in the arguments occur.
  \item As mentioned above, no condition on $\omega>0$ is required here.  If $\omega=1$, for example, then we have a curve that begins and ends at points along a ray.  Whether or not the initial endpoints coincide, we have convergence under rescaling to a circle.  If $\omega \in \mathbb{N}$ we have a spiral-like curve with $\omega$ revolutions converging under rescaling to an $\omega$-circle.  For more general $\omega>1$ we have an evolving spiral-like curve, converging under rescaling to a partially-multiply-covered circle.\\
\end{enumerate}    
\end{remark}

In addition to evolution equations used before for $\lambda$ constant, we will also need the following.\\

\begin{lemma}
 Under the flow \eqref{eq11}, with $\lambda = 0$, $$\frac{\mathrm{d} }{\mathrm{d} t} L^{4}\left ( t \right ) =-4\varepsilon (t)+2L^{3}\int _{\alpha }\left ( (k-\bar{k})^{4}+4(k-\bar{k})^{3}\bar{k}+6(k-\bar{k})^{2}\bar{k}^{2} \right )ds+32\omega^{4}\pi ^{4}.$$ 
 \end{lemma}
 
 Using the evolution $\int_\alpha k_s^2\, ds$ we obtain as in \cite{TW24}\\

\begin{lemma}
 Under the flow \eqref{eq11}, with $\lambda = 0$,
\begin{multline*}
\frac{\mathrm{d} }{\mathrm{d} t}\int _{\alpha }k_{s}^{2}\,ds\leq -\frac{1}{8}\int _{\alpha }k_{s^3}^{2}\,ds-\frac{13}{6}\bar{k}^{4}\int _{\alpha }k_{s}^{2}\,ds\\
-22\bar{k}^{3}\int _{\alpha }\left ( k-\bar{k} \right )k_{s}^{2}\,ds+5\int _{\alpha }\left ( k-\bar{k} \right )^{2}k_{ss}^{2}\,ds+10\bar{k}\int _{\alpha }\left ( k-\bar{k} \right )k_{ss}^{2}\,ds. 
\end{multline*}
\end{lemma}
Using Propositions \ref{psw} and \ref{eq266} we can obtain from the above that if $\varepsilon(0)$ is small enough, it will remain so. Our constants are slightly different from \cite{TW24} because our curve is not closed.  The workings are of course similar to the earlier sections but here we do not have the length bounded.\\

\begin{lemma}
 Under the flow \eqref{eq11}, with $\lambda = 0$, 
 $$ \frac{\mathrm{d} }{\mathrm{d} t}\int _{\alpha }k_{s}^{2}\,ds\leq -\left ( \frac{1}{8}-\frac{10}{\pi^{3} }\varepsilon (t)-(176\omega ^{3}+20\omega )\sqrt{\frac{2}{\pi ^{3}}}\sqrt{\varepsilon (t)} \right )\int _{\alpha }k_{s^3}^{2}\,ds-\frac{13}{6}\bar{k}^4\int _{\alpha }k_{s}^{2}\,ds.$$
 More precisely, at any time $t$ such that
$$L^{3}\left\|k_{s} \right\|_{2}^{2}:=\varepsilon (t)\leq \varepsilon _{\ast }(\omega):= \frac{\pi ^{3}}{200}\left ( \sqrt{(176\omega^{3}+20\omega )^{2}+\frac{5}{4}} -(176\omega^{3}+20\omega )\right )^{2}$$
 we obtain
 $$ \frac{\mathrm{d} }{\mathrm{d} t}\int _{\alpha }k_{s}^{2}\,ds\leq -\frac{1}{16}\int _{\alpha }k_{s^3}^{2}\,ds-\frac{13}{6}\bar{k}\int _{\alpha }k_{s}^{2}\,ds.$$
\end{lemma}
\mbox{}\\

Our next result, that when $\alpha$ is close to a circular arc, we have precise control on $L(t)$, is proven exactly as in \cite{TW24}[Proposition 3.6]. In our case, the constant $\hat{C}\left ( \beta, \omega  \right )$ is slightly different.\\

\begin{proposition}
Let $\beta>0.$ Let $\alpha :\left [ -1,\,1 \right ]\times \left [ 0,\,\infty  \right )\to \mathbb{R}^{2}$ to be a solution to  the free elastic flow with $\varepsilon (0)\leq \beta. $ Suppose that $$T_{\beta }:=inf\left\{t\geq 0~|\varepsilon (t)> \beta  \right\}\in \left [ 0,\,\infty  \right ].$$ Then, for all $t\in \left [ 0,\,T_{\beta } \right ),$
$$\left| L^{4}(t)-L^{4}(0)-32\omega ^{4}\pi ^{4}t\right|\leq \beta\, \hat{C}\left ( \beta,\,\omega  \right )t,$$
where $\hat{C}\left ( \beta,\,\omega  \right ) :=\frac{4}{\pi ^{3}}\beta +16\sqrt{\frac{2\omega ^{2}}{\pi ^{3}}}\sqrt{\beta }+48\omega ^{2}$.\\
\end{proposition}

This control on $L(t)$ allows us to control $\varepsilon(t)$ for a short time. Set
$$\delta_* (\varepsilon, \beta) := \frac{1}{\beta \hat{C}(\beta, \omega) + 32 \omega^4 \pi^4} \left[ \left( \frac{\beta}{\varepsilon} \right)^{\frac{4}{3}} - 1 \right].$$

We have exactly as in \cite{TW24}[Lemma 3.8]\\

\begin{lemma}
Let \(\beta \in (0, \varepsilon_*(\omega)]\). Suppose $\alpha (\cdot, t)$ solves  \eqref{eq11} with $\lambda =0$ and $\alpha_0$ satisfies $\varepsilon(0) \leq \beta$. Then, for all $t \in \left[ 0, L^4(t) \delta_* (\varepsilon(0), \beta) \right]$,$$\varepsilon(t) \leq \beta.$$
\end{lemma}

Now we introduce another scale-invariant quantity that decays under the flow \eqref{eq11} with $\lambda =0$. Specifically, define
$$\Gamma (t):=\frac{\int _{\alpha }k_{s}^{2}\,ds}{\left ( \int _{\alpha }k^{2}\,ds \right )^{3}}.$$

The same quantity was used in \cite{TW24}.  The proof of the next result is very similar to that in \cite{TW24} so we omit it, noting the only changes are use of the boundary conditions in `integration by parts', so the boundary terms are equal to zero, and the constants are slightly different in view of the differences of Propositions \ref{psw} and \ref{eq266} for curves with boundary as compared with closed curves.\\

\begin{lemma} Suppose $\alpha :\left [ -1,\,1 \right ]\times \left [ 0,\,\infty  \right )\to \mathbb{R}^{2}$ solves  \eqref{eq11} with $\lambda =0$ and $\varepsilon (0)\leq \varepsilon _{1}$.  
There exist $ \varepsilon _{1},c_{1},c_{2}> 0$ depending only $\omega $, such that if $T_{\varepsilon _{1}}:=\inf\left\{t\geq 0\,| \varepsilon (t)> \varepsilon _{1} \right\}\in \left ( 0,\,\infty  \right ]$  
then for $t\in \left [ 0,\,T_{\varepsilon _{1}} \right ),$
$$\Gamma (t)\leq \Gamma (0) \left ( 1+\frac{c_2}{L^{4}(0)}t \right )^{-c_{1}} \mbox{.}$$
\end{lemma}
\mbox{}\\

\begin{remark} In the above we may take explicitly $c_{1}=\frac{48\omega ^{4}\pi ^{4}}{5c_{2}}$ and $c_{2}=\varepsilon_{1} \hat{C}+32\omega ^{4} \pi ^{4}$.\\
\end{remark}

Together with the uniform lower length bound \eqref{E:llower}, the estimate on $\Gamma$ may be used iteratively, as in \cite{TW24} to establish a decaying bound $\varepsilon (t) $ for
all $t$.\\

\begin{corollary} There exists $ \varepsilon _{2}\in \left ( 0,\,\varepsilon _{1} \right )$,  depending only $\omega $, such that if  $\alpha :\left [ -1,\,1 \right ]\times \left [ 0,\,\infty  \right )\to \mathbb{R}^{2}$ has $\varepsilon (0)\leq \varepsilon _{1}$ and solves  \eqref{eq11} with $\lambda =0$ then for $t\in \left [ 0,\,\infty  \right ),$ we have $\varepsilon (t)\leq \varepsilon _{1}.$  Furthermore
$$\varepsilon  (t)\leq c_{3} \left ( 1+\frac{t}{L^{4}(0)} \right )^{-c_{1}},$$ where $c_{1},\,c_{3} > 0$  depend only on $\omega$. In particular, $ \varepsilon (t)\to 0$ as $t\to \infty .$\\
\end{corollary}

Standard interpolation arguments as per \cite{DKS02} but in the case with boundary provide the following.\\

\begin{lemma} 
 Under the flow \eqref{eq11}, with $\lambda = 0$. Then for all $\ell \in \mathbb{N},$ 
 $$\frac{\mathrm{d} }{\mathrm{d} t}\int _{\alpha }k_{s^{\ell}}^{2}\,ds+\int _{\alpha }k_{s^{\ell+2}}^{2}\,ds\leq C_{\ell}\left ( \int _{\alpha }k^{2}\,ds \right )^{2\ell+5}.$$
\end{lemma}
Now we can show that the scale-invariant quantity $L(t)k(\cdot , t)$
approaches uniformly a constant. This means that after scaling
for length, the solution approaches a circular arc.\\

\begin{corollary}  \label{T:kdecay}
Suppose that  $\alpha :\left [ -1,\,1 \right ]\times \left [ 0,\,\infty  \right )\to \mathbb{R}^{2}$ \eqref{eq11} with $\lambda =0$ and $\varepsilon (0)\leq \varepsilon_2.$ Then, for some $C = C\left ( \omega, L(0) \right )>0$, $$\left\|L(t)k\left ( \cdot,\,t \right )-2\omega \pi  \right\|_{\infty }\leq C\left ( 1+t \right )^{-\frac{c_{1}}{2}},$$
and, moreover,
$$\left\|k^{3}\left ( \cdot,\,t  \right ) \right\|_{\infty }\leq C\left ( 1+t \right )^{-\frac{(3+2c_{1})}{4}}+\frac{8\omega ^{3}\pi ^{3}}{\left ( L^{4}(0)+(32\omega ^{4}\pi ^{4}-\varepsilon _{1}\hat{C}(\varepsilon _{1},\,\omega ))t \right )^{\frac{3}{4}}}.$$
\end{corollary} 
\mbox{}\\
Using now the curvature decay and bounds on $\int k_{s^{\ell}}^{2}\,ds$ it is standard to obtain decay of all curvature derivatives.\\

\begin{corollary}Let  $\alpha :\left [ -1,\,1 \right ]\times \left [ 0,\,\infty  \right )\to \mathbb{R}^{2}$ satisfy \eqref{eq11} with $\lambda =0$ and with $\alpha_0$ satisfying $\varepsilon (0)\leq \varepsilon_2.$ Then, for any integer $\ell \geq 1,$
$$\left\|k_{s^{\ell}} \right\|_{\infty }\leq C\left ( 1+t \right )^{-\frac{(\ell+1+c_{1})}{4}},$$
where $C$ depends on $\ell,\, \omega,\, L(0),\,$ and $\left\|k_{s^{2\ell+1}} \right\|_{2 }^{2}(0).$
\end{corollary} 
Finally, to prove smooth convergence under rescaling to the circular arc, we need to control the rescaled embedding $\eta \left ( \cdot, t \right )=\displaystyle{\frac{1}{L(t)}\alpha \left ( \cdot, t \right )}$ and its derivatives.  We know already from Corollary \ref{T:kdecay} that $k^{(\eta)}(\cdot, t) \to 2\omega \pi $ pointwise. Moreover, using the length control of $\eta$,

$$\left\| k^{\eta}_{s^{\ell}} \right\|_{\infty} = L^{m+1} \left\| k_{s^{\ell}} \right\|_{\infty} \leq C (1 + t)^{-\frac{c_1}{4}} \to 0.$$

The boundary conditions ensure that the arc has centre at the cone tip.\\

The spatial and mixed derivatives of $\eta$ are now controlled using standard techniques, converting derivatives with respect to $s$ back to those with respect to $u \in [-1, 1]$ and using the existing estimates.  Convergence may be upgraded to full exponential convergence via a standard linearisation argument. This completes the proof of Theorem \ref{T:main3}.\hspace*{\fill}$\Box$






\backmatter

\bmhead{Acknowledgements}

The research of the first author was supported by a postgraduate scholarship from the Department of Mathematics, College of Science, Imam Abdulrahman Bin Faisal University, P. O. Box 1982, Dammam, Saudi Arabia.  Part of this research was completed while the second author was visiting the University of Science and Technology, China under a Chinese Academy of Sciences Presidents' International Fellowship Initiative visiting fellowship, grant number 2024PVA0042.  The authors are grateful for this support.

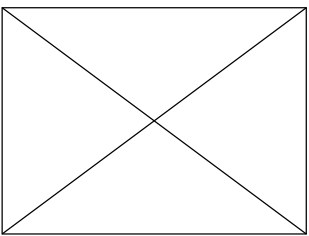
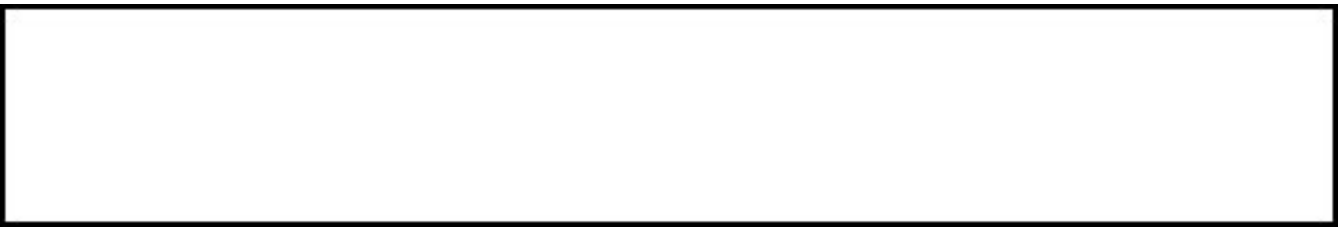
\end{document}